\documentclass{amsart}
\usepackage{amsmath}
\usepackage{amssymb}
\usepackage{color}
\usepackage{mathdots}

\usepackage{graphicx}
\usepackage{epsfig}
\usepackage{multirow}
\usepackage{colortbl}

\definecolor{lightgray}{gray}{0.9}

\DeclareMathOperator{\Id}{Id}

\def\FF{\mathbb{F}}

\def\KK{\mathbb{K}}

\newtheorem{theorem}{Theorem}[section]

\newtheorem{proposition}[theorem]{Proposition}
\newtheorem{corollary}[theorem]{Corollary}
\theoremstyle{definition}
\newtheorem{definition}[theorem]{Definition}

\theoremstyle{remark}

\numberwithin{equation}{section}

\def\separa{\hbox to 14 truecm{\hrulefill}}

\date\today

\author[P. Danchev]{Peter Danchev}
\address{Institute of Mathematics and Informatics, Bulgarian Academy of Sciences, 1113 Sofia, Bulgaria}
\email{danchev@math.bas.bg}
\thanks{The first author was partially supported  by the BIDEB 2221 of T\"UB\'ITAK}

\author[E. Garc\'\i a]{Esther Garc\'\i a}
\address{Departamento de Matem\'{a}tica  Aplicada, Ciencia e Ingenier\'{\i}a de  Materiales y Tecnolog\'{\i}a Electr\'onica,
Universidad Rey Juan Carlos, 28933 M\'{o}s\-to\-les (Madrid), Spain}
\thanks{The second author and third authors were partially supported by B42025-003: Ayudas para Proyectos Puente de la UMA}
\email{esther.garcia@urjc.es}

\author[M. G\'omez Lozano]{Miguel G\'omez Lozano}
\address{Departamento de \'Algebra, Geometr\'{\i}a y
Topolog\'{\i}a, Universidad de M\'alaga, 29071 M\'alaga, Spain}
\thanks{The three authors were partially supported by the Junta de Andaluc\'{\i}a PPRO-FQM264-G-2023 (FQM264-G-FEDER)}
\email{miggl@uma.es}

\begin{document}

\title[Matrices over finite fields of characteristic 2 ...]{Matrices over finite fields of characteristic 2 \\ as sums of diagonalizable and square-zero matrices}

\maketitle

\begin{abstract} We investigate the problem asking when any square matrix whose entries lie in a finite field of characteristic 2 is decomposable into the sum of a diagonalizable matrix and a nilpotent matrix with index of nilpotency at most 2 and, as a result, we completely resolve this question in the affirmative for any finite field of characteristic 2 having strictly more than three elements. Our main theorem of that type, combined with results from our recent publication in Linear Algebra \& Appl. (2026) (see \cite{DGL4}), totally settle this problem for all finite fields different from  $\FF_2$ and $\FF_3$. However, in this paper we also prove that each matrix over $\FF_2$ is expressible as the sum of a potent matrix with index of potency not exceeding 4 and a nilpotent matrix with index of nilpotency not exceeding 2, thus substantiating recent examples due to \v{S}ter in Linear Algebra \& Appl. (2018) and Shitov in Indag. Math. (2019) (see, respectively, \cite{St1} and \cite{Sh}).
\end{abstract}


\bigskip




\section{Fundamentals and Motivation}

All matrices in the current paper are {\it square} having the same number of columns and rows, and we will refer to this number as its order. Such a matrix is termed {\it non-derogative} (or, in other terms, {\it non-derogatory}) if its minimal polynomial is identical with its characteristic polynomial; in such case it is well known that under an appropriate change of basis it has the form of the companion matrix of its characteristic polynomial. Also, a matrix $M$ is called {\it diagonalizable} (or, in other words, {\it non-defective}) if it is similar to a diagonal matrix, that is, there is an invertible matrix $U$ such that the matrix $U^{-1}MU$ is diagonal. On the other hand, a {\it nilpotent} matrix $N$ is the one for which $N^k = 0$ for some integer $k>0$. In the case $k=2$, such a matrix is just said to be {\it square-zero} for short. Contrastingly, a {\it $t$-potent} matrix $Z$ is the one for which $Z^t=Z$ for some integer $t\geq 2$. In the case when $t=2$, such a matrix is just said to be {\it idempotent}.

\medskip

The problem of decomposing a given matrix over an arbitrary field into the sum of an idempotent matrix and a nilpotent matrix originated from \cite{BCDM} by establishing that such a decomposition occurs only when the field is finite and has exactly two elements, say $\mathbb{F}_2$. This basic result was considerably strengthened in a series of further works of this sort like these obtained in \cite{St1}, \cite{St2} and \cite{Sh}, respectively.

\medskip

Besides, in 2018, Breaz explored in \cite{Br} those matrices over finite fields that are a sum of a periodic matrix and a nilpotent matrix by demonstrating that every matrix over a finite field of odd cardinality $q$ is a sum of a $q$-potent matrix and a nilpotent matrix of order less than or equal to $3$ (see \cite{BM} too). Moreover, he also asked in \cite[Remark 7]{Br} whether for matrices of a sufficiently large order over a finite field of odd power $q$ there will exist a more precise decomposition of the form of a $q$-potent matrix plus a square-zero matrix. This difficult question was fully resolved recently in \cite{DGL4} by proving that for some finite fields of cardinality 3 and order multiple of 3 the answer is unfortunately {\it negative}, whereas for fields of cardinality $q$ at least 5 the answer is absolutely {\it positive} (see, for a more detailed information, \cite{DGL1}, \cite{DGL2} and \cite{DGL3} as well).

\medskip

Recall that, in \cite{St1}, ~\v{S}ter showed that the following two $4\times 4$ matrices

\medskip

$$
A=\left(
    \begin{array}{cccc}
      0 & 0 & 0 & 1 \\
      1 & 0 & 0 & 0 \\
      0 & 1 & 0 & 0 \\
      0 & 0 & 1 & 1 \\
    \end{array}
  \right),~ B=\left(
    \begin{array}{cccc}
      0 & 0 & 0 & 1 \\
      1 & 0 & 0 & 1 \\
      0 &1 & 0 &1 \\
      0 & 0 &1 & 1 \\
    \end{array}
  \right)
$$

\medskip

\noindent{\it cannot} be expressed as the sum of an idempotent matrix and a nilpotent matrix $N$ with $N^3=0$ (notice that matrices over $\mathbb{F}_2$ are idempotent if, and only if, they are diagonalizable).
In this vein, Breaz additionally asked in \cite{Br} of whether or not for big enough positive integers $n$ all matrices of order $n$ over a field of even cardinality $q$ admit decompositions of the form $E+N$ with $E^q=E$ and $N^3=0$. Subsequently, Shitov answered in \cite{Sh} this Breaz's question by showing that odd direct sums of matrices of the form $A$ cannot be expressed as the sum of an idempotent matrix and a nilpotent matrix of index lower than 4 using an ingenious trick concerning minimal polynomials and ranks of matrices.

\medskip

Nevertheless, this challenging problem remained still open for matrices over fields of even characteristic and at least four elements. So, this is actually the main motivation of the present paper and, especially, the aim is twofold: on the one hand, to complete in all generality our work begun in \cite{DGL1} by showing that every matrix over a field of characteristic 2 and at least four elements can always be expressed as the sum of a diagonalizable matrix and a square-zero matrix, thus finishing all the groundwork;  on the other hand, to answer Breaz's question positively by establishing that every matrix over a field of even cardinality $q$ with $q\ge 4$ can always be expressed as $E+N$ such that $E^q=E$ and $N^2=0$. To achieve these two goals, our structural results are stated and proved in Theorem~\ref{main} and Proposition~\ref{adjusting}, respectively. We also succeeded in Corollary \ref{F2-potent} to express every matrix over $\FF_2$ as the sum of a potent matrix having order of potency less than or equal to 4 and a square-zero matrix, thus paralleling the aforementioned examples constructed by \v{S}ter \cite{St1} and Shitov \cite{Sh}.

\section{Non-derogative matrices over a field $\FF$ of characteristic 2, $\#\FF\ge4$}

The principal target of this section is to show that every non-derogative matrix can be expressed as the sum of a square-zero matrix and a diagonalizable matrix. Before doing that, let us first list some non-derogative diagonalizable matrices of small orders that will appear in our arguments presented below.

\medskip

We use hereafter the notation $\#\FF$ to denote the cardinality of the (finite) field $\FF$. Moreover, along the text without explicit further mention, all matrices of order $n$ over $\FF$ will be identified with endomorphisms of $\FF^n$, and we will denote by $\{e_1,\dots, e_n\}$ the {\it canonical basis} of $\FF^n$.

\medskip

\begin{definition}\label{notation} Let $\mathbb{F}$ be a field of even characteristic with $\#\mathbb{F}\ge 4$, let $a\in \mathbb{F}$ with $a\ne 0,1$, and let us consider the following matrices of order 2, 3 and 4 over $\FF$:

\medskip

\noindent $\bullet$ $n=2$: for any $0\ne u_1\in \FF$, we set

$$D_{2}:=\left(
    \begin{array}{cc}
      0 & 0 \\
      1 & u_1 \\
    \end{array}
  \right)$$
with associated characteristic polynomial given by $x(x+u_1)$ and basis of eigenvectors $$\mathcal{B}=\{u_1 e_1+e_2, e_2\}.$$

\medskip

\noindent $\bullet$ $n=3$: for any $u_2\in \FF$ and any $a\in \FF$ with $a\ne u_2, u_2+1$, we set
$$
D_{3}(u_2):=
\left(
  \begin{array}{ccc}
     0 & 0 & (a^2+a)(u_2+1) \\
     1 & 0 & u_2+a^2+a+1 \\
     0 & 1 & u_2 \\
  \end{array}
\right)
         ,$$
where $D_3(u_2)$ is diagonalizable with eigenvalues $u_2+1, a, a+1$ and basis of eigenvectors
$$\mathcal{B}=\{ (a^2+a)e_1+e_2+e_3, ((a+1)u_2+a+1)e_1+(u_2+a)e_2+e_3, (au_2+a)e_1+(a+u_2+1)e_2+e_3
\}.$$
In the case when $u_2=0$, we will denote $D_3:=D_3(0)$, i.e.,
$$D_{3}:=
\left(
  \begin{array}{ccc}
     0 & 0 & a^2+a \\
     1 & 0 & a^2+a+1 \\
     0 & 1 & 0 \\
  \end{array}
\right)
$$
with associated characteristic polynomial given by $(x+1)(x+a)(x+a+1)$ and basis of eigenvectors
$$
\mathcal{B}=\{(a^2+a)e_1+e_2+e_3, (a+1)e_1+ae_2+e_3, ae_1+(a+1)e_2+e_3\}.
$$
We will also use in the sequel the square-zero matrix
$$
N_3:=\left(
           \begin{array}{ccc}
             0 & 0 & a^2+a \\
             0 & 0 & a^2+a+1 \\
             0 & 0 & 0 \\
           \end{array}
         \right).
$$

\medskip

\noindent $\bullet$ $n=4$: for any $a\in \FF$ with $a\ne 0,1$, we set

$$D_{4}:=\left(
           \begin{array}{cccc}
             0 & 0 & 0 & 0 \\
             1 & 0 & 0 & a^2+a \\
             0 & 1 & 0 & a^2+a+1 \\
             0 & 0 & 1 & 0 \\
           \end{array}
         \right),\qquad N_{4}:=\left(
           \begin{array}{cccc}
             0 & 0 & 0 & 0 \\
             0 & 0 & 0 & a^2+a \\
             0 & 0 & 0 & a^2+a+1 \\
             0 & 0 & 0 & 0 \\
           \end{array}
         \right),$$
where $D_4$ is diagonalizable with associated characteristic polynomial given by $x(x+1)(x+a)(x+a+1)$ and basis of eigenvectors
$$
\mathcal{B}=\{(a^2+a)e_1+(a^2+a+1)e_2+e_4, (a^2+a)e_2+e_3+e_4, (a+1)e_2+ae_3+e_4, ae_2+(a+1)e_3+e_4\}.
$$

\medskip

\noindent $\bullet$ Moreover, for every $r,s\in \mathbb{N}$, let us define
$$Q_{r,s}:=\left(
            \begin{array}{cccc}
              0 & \cdots & 0 & 1 \\
              0 & \cdots  & 0 & 0 \\
              \vdots & \ddots  & \vdots & \vdots \\
              0 &  \cdots & 0 & 0 \\
            \end{array}
          \right)\in M_{r,s}(\mathbb{F}).$$
\end{definition}

A useful technical claim asserts the following.

\begin{proposition}\label{case2}
Let $\FF$ be a field of characteristic 2, and let $B$ be a non-derogative matrix over $\FF$ of order 2 and trace $u_1$. Then, $B$ can be expressed as the sum of a square-zero matrix and a diagonalizable matrix $D$, where if $u_1=0$, the eigenvalues of $D$ are $v$ (double) for some $v$ with $v^2=u_0$, and if $u_1\ne 0$, the eigenvalues of $D$ are $0, u_1$. In particular, if $\FF=\FF_2$, then $B$ can be expressed as the sum of a square-zero matrix and a potent matrix $D$ such that $D^2=D$.
\end{proposition}

\begin{proof}
Without loss of generality, we can suppose that $B$ is of the form
 $$B=\left(
  \begin{array}{cccc}
    0 & u_0 \\
    1 & u_1  \\
  \end{array}
\right).$$

If $u_1=0$, then since the map $\phi:\FF\to \FF$ defined by $\phi(x)=x^2$ is a one-to-one $\FF_2$-linear map, it is surjective and, therefore, for each $u_0\in \FF$, there exists $v\in \FF$ with $v^2=u_0$. Hence,
$$B=\left(
     \begin{array}{cc}
       0 & v^2 \\
       1 & 0 \\
     \end{array}
   \right)=\underbrace{\left(
     \begin{array}{cc}
       v & v^2 \\
       1 & v \\
     \end{array}
   \right)}_N+\underbrace{\left(
     \begin{array}{cc}
       v & 0 \\
       0 & v \\
     \end{array}
   \right)}_D$$
with $N^2=0$ and $D$ diagonal. In particular, if $\FF=\FF_2$, then $D$ manifestly satisfies $D^2=D$.

If, however, $u_1\ne 0$, then we can decompose $B$ into $N+D$, where $N$ is a square-zero rank one matrix and $D$ is diagonalizable such that
$$B=\left(
     \begin{array}{cc}
       0 & u_0 \\
       1 & u_1 \\
     \end{array}
   \right)=\underbrace{\left(
     \begin{array}{cc}
       0 & u_0 \\
       0 & 0 \\
     \end{array}
   \right)}_N+\underbrace{\left(
     \begin{array}{cc}
       0 & 0 \\
       1 & u_1 \\
     \end{array}
   \right)}_D,$$
as suspected. Particularly, if $\FF=\FF_2$, then $u_1=1$ and so
$D=\left(
                                                                                    \begin{array}{cc}
                                                                                      0 & 0 \\
                                                                                      1 & 1 \\
                                                                                    \end{array}                                                                                  \right)$
obviously satisfies $D^2=D$.
\end{proof}

Our next helpful technicality claims the following.

\begin{proposition}\label{case3}
Let $\FF$ be a field of characteristic 2 with $\#\FF\ge4$, and let $B$ be a non-derogative matrix over $\FF$ of order 3 and trace $u_2$. Then, for any $a\in \mathbb{F}$, $a\ne u_2,u_2+1$, $B$ can be expressed as the sum of a square-zero matrix and a diagonalizable matrix with eigenvalues $a,a+1,u_2+1$.
\end{proposition}

\begin{proof}
We can follow the same argument as given in \cite[Lemma 3.1]{DGL1} to decompose $B$ into $N+D$, where $N$ is a square-zero matrix of rank one and $D$ is diagonalizable with 3 different eigenvalues.

In fact, we know that $B$ is similar to a matrix of the form
 $$A=\left(
  \begin{array}{ccc}
     0 & 0 & u_0 \\
     1 & 0 & u_1 \\
     0 & 1 & u_2 \\
  \end{array}
\right).$$
Since $a\ne u_2,u_2+1$, we can consider the three different elements $a,a+1,u_2+1\in \mathbb{F}$. Hence,
$$\underbrace{\left(
  \begin{array}{ccc}
     0 & 0 & u_0 \\
     1 & 0 & u_1 \\
     0 & 1 & u_2 \\
  \end{array}
\right)}_A=\underbrace{\left(
  \begin{array}{ccc}
     0 & 0 & (a^2+a)(u_2+1)+u_0 \\
     0 & 0 & u_2+a^2+a+1+u_1 \\
     0 & 0 & 0 \\
  \end{array}
\right)}_N+\underbrace{\left(
  \begin{array}{ccc}
     0 & 0 & (a^2+a)(u_2+1) \\
     1 & 0 & u_2+a^2+a+1 \\
     0 & 1 & u_2 \\
  \end{array}
\right)}_D,$$
where $N$ is a rank one square-zero matrix and $D$ is diagonalizable with eigenvalues $u_2+1, a, a+1$ -- see Definition~\ref{notation}.
\end{proof}

As a crucial consequence, we yield:

\begin{corollary}\label{case3-corollary}
If $\FF=\FF_2$ and $B$ is a non-derogative matrix over $\FF$ of order 3, then $B$ can be expressed as the sum of a square-zero matrix and a potent matrix $D$ with $D^4=D$.
\end{corollary}

\begin{proof}
We know that $B$ is similar to a matrix of the form
$$A=\left(
  \begin{array}{ccc}
     0 & 0 & u_0 \\
     1 & 0 & u_1 \\
     0 & 1 & u_2 \\
  \end{array}
\right)\in M_3(\FF_2)\subseteq M_3(\FF_4).$$
Let $a\in \FF_4$ be the element such that $a^2+a+1=0$; then, the decomposition given in Proposition \ref{case3} for $A$ reads as follows: $A=N+D$, where
$$
N:=\left(
  \begin{array}{ccc}
     0 & 0 & (a^2+a)(u_2+1)+u_0 \\
     0 & 0 & u_2+a^2+a+1+u_1 \\
     0 & 0 & 0 \\
  \end{array}
\right)=\left(
  \begin{array}{ccc}
     0 & 0 & u_2+1+u_0 \\
     0 & 0 & u_2+u_1 \\
     0 & 0 & 0 \\
  \end{array}
\right)\in M(\FF_2)
$$
\noindent and
$$
D:=\left(
  \begin{array}{ccc}
     0 & 0 & (a^2+a)(u_2+1) \\
     1 & 0 & u_2+a^2+a+1 \\
     0 & 1 & u_2 \\
  \end{array}
\right)=\left(
  \begin{array}{ccc}
     0 & 0 & u_2+1 \\
     1 & 0 & u_2 \\
     0 & 1 & u_2 \\
  \end{array}
\right)\in M(\FF_2).
$$
Clearly, $N^2=0$ and $D^4=D$, because it is diagonalizable matrix over $\FF_4$ with eigenvalues $a,a+1,u_2+1$ exploiting Proposition \ref{case3} again.
\end{proof}

We now study non-derogative matrices of order 4.

\begin{proposition}\label{case4}
Let $\FF$ be a field of characteristic 2 with $\#\FF\ge4$, and let $B$ be a non-derogative matrix over $\FF$ of order 4.
\begin{itemize}
\item[(a)] If $\text{Trace}(B)=u_3\ne 0$, then, for any $a\in\FF$ such that $a\ne 0, u_3$, $B$ can be expressed as the sum of a square-zero matrix and a diagonalizable matrix with eigenvalues $0,a,u_3+a$.
\item[(b)] If $\text{Trace}(B)=0$, then, for any  $a\in \FF$ such that $a\ne 0,1$, $B$ can be expressed as the sum of a square-zero matrix and a diagonalizable matrix with eigenvalues $0,1,a,a+1$.
\end{itemize}
\end{proposition}

\begin{proof}
{\bf Case (a):} If $\text{Trace}(B)=u_3\ne 0$, we know that $B$ is similar to a matrix of the form
$$A=\left(
  \begin{array}{cccc}
    0 & 0 & 0 & u_0 \\
    1 & 0 & 0 & u_1 \\
    0 & 1 & 0 & u_2 \\
   0 & 0 & 1 & u_3 \\
  \end{array}
\right).$$ Choose $a\in\FF$ such that $a\ne 0, u_3$, and consider the square-zero matrix
$$N:=\left(
  \begin{array}{cccc}
    0 & 0 & 0 & 0 \\
    1 & 0 & 0 & 0 \\
    0 & 0 & a+u_3 & a^2+u_3^2 \\
   0 & 0 & 1 & a+u_3 \\
  \end{array}
\right).$$
So, $A=N+D$ for
$$D:= \left(
  \begin{array}{cccc}
    0 & 0 & 0 & u_0 \\
    0 & 0 & 0 & u_1 \\
    0 & 1 & a+u_3 & u_2+a^2+u_3^2 \\
   0 & 0 & 0 & a \\
  \end{array}
\right),$$ where $D$ is diagonalizable with eigenvalues $0,0,a+u_3,a$ and basis of eigenvectors $$\{e_1, (a+u_3) e_2+ e_3, e_3, u_0 u_3 e_1+u_1u_3 e_2+ (au_3^2+au_2+u_1+a^3)e_3+a u_3 e_4\}.$$

{\bf Case (b):} If $\text{Trace}(B)=0$ for every $a \in \FF$ such that $a\ne 0,1$, we can follow the same argument as given in \cite[Lemma 3.1]{DGL1} to decompose $B$ into $N+D$, where $N$ is a square-zero matrix of rank one and $D$ is diagonalizable with 4 different eigenvalues $0,1,a,a+1$.
\end{proof}

A critical consequence is as follows:

\begin{corollary}\label{case4-corollary}
If $\FF=\FF_2$ and $B$ is a non-derogatory matrix of order 4 over $\FF$, then $B$ can be expressed as the sum of a square-zero matrix and a potent matrix $D$ with $D^4=D$.
\end{corollary}

\begin{proof} We distinguish two basic cases as follows.

\medskip

{\bf Case (a):} If $\text{Trace}(B)=1$, let us consider the rational canonical decomposition of $B$ with respect to its invariant factors.

\medskip

\noindent -- If all invariant factors have degree less than or equal to three, their corresponding companion matrices can be expressed as the sum of a square-zero matrix and a potent matrix of index less than or equal to 4 in conjunction with Proposition \ref{case2} and Corollary \ref{case3-corollary}.

\medskip

\noindent -- If there is only one invariant factor of degree 4, it is necessarily of the form $x^4+x^3+1$ or $x^4+x^3+x^2+x+1$ (notice that it {\it cannot} be the square of the irreducible polynomial of degree 2 over $\FF_2$, because it has trace equal to one). Suppose that $B$ is similar to the companion matrix of $x^4+x^3+1$, i.e., $B$ is similar to a matrix of the form
$$
A=\left(
    \begin{array}{cccc}
      0 & 0 & 0 & 1 \\
      1 & 0 & 0 & 0 \\
      0 & 1 & 0 & 0 \\
      0 & 0 & 1 & 1 \\
    \end{array}
  \right).
$$
Then,
$$
\underbrace{\left(
    \begin{array}{cccc}
      0 & 0 & 0 & 1 \\
      1 & 0 & 0 & 0 \\
      0 & 1 & 0 & 0 \\
      0 & 0 & 1 & 1 \\
    \end{array}
  \right)}_A=\underbrace{\left(
    \begin{array}{cccc}
      1 & 1 & 0 & 0 \\
      1 & 1 & 0 & 0 \\
      0 & 0 & 0 & 1 \\
      0 & 0 & 0 & 0 \\
    \end{array}
  \right)}_N+\underbrace{\left(
    \begin{array}{cccc}
      1 & 1 & 0 & 1 \\
      0 & 1 & 0 & 0 \\
      0 & 1 & 0 & 1 \\
      0 & 0 & 1 & 1 \\
    \end{array}
  \right)}_D,
$$
where $N^2=0$ and $D^4=D$, because $D$ is diagonalizable over $\FF_4$ with eigenvalues $1,1, a$ and $a+1$, where $a\in \FF_4$ satisfies $a^2+a+1=0$.

If $B$ is similar to the companion matrix of $x^4+x^3+x^2+x+1$, let us consider the matrix $B+\Id$, which is again non-derogative with characteristic polynomial equal to $x^4+x^3+1$. Hence, $B+\Id$ is the sum of a square-zero matrix $N'$ and a diagonalizable matrix $D'$ over $\FF_4$ with eigenvalues $1,1, a$ and $a+1$, so that $B=N'+(D'+\Id)$, where $D'+\Id$ has eigenvalues $0,0, a, a+1$, whence $(D'+\Id)^4=D'+\Id$.

\medskip

{\bf Case (b):} If $\text{Trace}(B)=0$, we know that $B$ is similar to a matrix of the form
$$
A=\left(
    \begin{array}{cccc}
      0 & 0 & 0 & u_0 \\
      1 & 0 & 0 & u_1 \\
      0 & 1 & 0 & u_2 \\
      0 & 0 & 1 & 0 \\
    \end{array}
  \right).
$$ Then,
$$
\underbrace{\left(
    \begin{array}{cccc}
      0 & 0 & 0 & u_0 \\
      1 & 0 & 0 & u_1 \\
      0 & 1 & 0 & u_2 \\
      0 & 0 & 1 & 0 \\
    \end{array}
  \right)}_A=\underbrace{\left(
    \begin{array}{cccc}
      0 & 0 & 0 & u_0 \\
      0 & 0 & 0 & u_1+1 \\
      0 & 0 & 0 & u_2 \\
      0 & 0 & 0 & 0 \\
    \end{array}
  \right)}_N+\underbrace{\left(
    \begin{array}{cccc}
      0 & 0 & 0 & 0 \\
      1 & 0 & 0 & 1 \\
      0 & 1 & 0 & 0 \\
      0 & 0 & 1 & 0 \\
    \end{array}
  \right)}_D,
$$
where $N^2=0$ and $D$ is diagonalizable with eigenvalues $0,1,a,a+1$ for $a\in \FF_4$ such that $a^2+a+1=0$; hence $D^4=D$.
\end{proof}

Now, we examine non-derogative matrices of order $4k$, where $k\ge 2$.

\begin{proposition}\label{4k}
Let $\FF$ be a field of characteristic 2 with $\#\FF\ge4$, and let $B$ be a non-derogative matrix over $\FF$ of order $4k$ with $k\ge 2$. Then, for every $a\in \FF$ with $a\ne 0,1$, we have:
  \begin{itemize}
\item[(a)] If $\text{Trace}(B)=c\ne 0$, there exists a square-zero matrix $N\in {M}_{4k}(\FF)$ such that $B=N+D$, where $D$ is diagonalizable with eigenvalues $0, c, ca,c(a+1)$.
\item[(b)] If $\text{Trace}(B)=0$, there exists a square-zero matrix $N\in {M}_{4k}(\FF)$ such that $B=N+D$, where $D$ is diagonalizable with eigenvalues $0,1,a$ and $a+1$.
  \end{itemize}
\end{proposition}

\begin{proof}
{\bf Case (a):} Suppose that $\text{Trace}(B)=c\ne 0$. Without loss of generality, let us replace $B$ by $A=c^{-1}B$, which is again a non-derogative matrix which possesses trace equal to 1; thereby, we can suppose without loss of generality that $A$ has the form
$$
A=\left(\begin{array}{cccccc}
                                                                          0 & 0 & \cdots & 0 & 0 &u_0\\
                                                                          1 & 0 & \cdots & 0 & 0 &u_1\\
                                                                          0 & 1 & \cdots & 0 & 0 &u_2\\
                                                                          \vdots & \vdots & \ddots & \vdots & \vdots &\vdots\\
                                                                          0 & 0 & \cdots & 1 & 0 &u_{4k-2}\\
                                                                          0 & 0 & \cdots & 0 & 1 &1
                                                                        \end{array}\right).$$
We are now going to build a square-zero matrix $N$ such that $A=N+D$, where $D$ can be expressed as the direct sum of $k-2$ diagonalizable blocks of order $4$, two diagonalizable blocks of order $3$ and one diagonalizable block of order $2$. To that goal, let us consider
$$
M_s:=\left(
   \begin{array}{cc}
      0 &u_{4s-4}\\
      0 &u_{4s-3}\\
      0 &u_{4s-1}+u_{4s-2}\\
      0 &0\\
   \end{array}
 \right) \hbox{ for  $s=1,\dots, k-2$,}
$$
\small
 $$M_{k-1}:=\left(
   \begin{array}{cc}
      0 &u_{4k-8}\\
      0 &u_{4k-6}+u_{4k-7}\\
      0 &0\\
   \end{array}
 \right),\  M_{k}:=\left(
   \begin{array}{cc}
      0 &u_{4k-5}\\
      0 &u_{4k-3}+u_{4k-4}\\
      0 &0\\
   \end{array}
 \right),\ M_{k+1}:=\left(
                  \begin{array}{cc}
                    0 & u_{4k-2} \\
                    0 & 0 \\
                  \end{array}
                \right),$$\normalsize
and put
$$N:=\left(
      \begin{array}{cccccccc}
        N_4  & 0 &\cdots & 0& 0 & 0 & 0&  M_1 \\
        Q_{4,4}  &N_4 & \cdots & 0 & 0 & 0& 0  & M_2 \\
        \vdots &\ddots & \ddots & \vdots& \vdots & \vdots& 0  & \vdots \\
        0 &0 & \ddots & N_4 & 0& 0 & 0&  M_{k-3} \\
        0 &0 & \cdots & Q_{4,4} & N_4& 0& 0 &  M_{k-2} \\
        0 &0 & \cdots & 0 & Q_{3,4}& N_3& 0 &  M_{k-1} \\
        0 &0 & \cdots & 0 & 0& Q_{3,3} &N_3&   M_k \\
        0 &0 & \cdots & 0 &0 & 0& Q_{2,3} &  M_{k+1} \\
      \end{array}
    \right)\in M_{4k}(\mathbb{F}).
$$
It can be verified that $N_4 M_s= Q_{4,4} M_s= M_{3,4} Q_s=0$ for $s=1\dots, k-2$, $N_3 M_{k-1}=N_3 M_{k}=Q_{3,3} M_{k-1}= Q_{2,3} M_k=0$, and $M_{k+1}M_{k+1}=0$, so we deduce that $N$ is a square-zero matrix.

Moreover, if we denote
$$T_s:=\left(
   \begin{array}{cc}
       0 &0\\
       0 &0\\
       0 &u_{4s-1}\\
       0 &u_{4s-1}\\
   \end{array}
 \right),\hbox{ for $s=1\dots, k-2$,}
$$
$$T_{k-1}:=\left(
   \begin{array}{cc}
        0 & 0\\
        0 &u_{4k-6}\\
        0 &u_{4k-6}\\
   \end{array}
 \right),\ T_{k}:=\left(
   \begin{array}{cc}
        0 & 0\\
        0 &u_{4k-3}\\
        0 &u_{4k-3}\\
   \end{array}
 \right),
 $$
we have that $A=N+D$, where
$$D:=\left(
         \begin{array}{cccccc}
           D_4 & \cdots & 0 & 0 & 0 & T_1 \\
           \vdots & \ddots & \vdots & \vdots & \vdots & \vdots \\
           0 & \cdots & D_4 & 0 & 0 & T_{k-2} \\
           0 & \cdots & 0 & D_3 & 0 & T_{k-1} \\
           0 & \cdots & 0 & 0 & D_3 & T_k \\
           0 & \cdots & 0 & 0 & 0 & D_2 \\
         \end{array}
       \right).$$
Furthermore, since $D_4$ is diagonalizable, for any $s=1\dots, k-2$, the vectors
\begin{align*}
 v_{4s-3}&=(a^2+a)e_{4s-3}+(a^2+a+1)e_{4s-2}+e_{4s},\\
 v_{4s-2}&=(a^2+a)e_{4s-2}+e_{4s-1}+e_{4s},\\
 v_{4s-1}&=(a+1)e_{4s-2}+a e_{4s-1}+e_{4s},\\
 v_{4s}&=a e_{4s-2}+(a+1) e_{4s-1}+e_{4s}
\end{align*}
form a set of $4k-8$ independent eigenvectors of $D$ associated to the eigenvalues $0,1,a,a+1$. Moreover, since $D_3$ is diagonalizable for $r=0,1$ the vectors
\begin{align*}
 v_{4k-7+3r}&=(a^2+a)e_{4k-7+3r}+e_{4k-6+3r}+e_{4k-5+3r},\\
 v_{4k-6+3r}&=(a+1)e_{4k-7+3r}+a e_{4k-6+3r}+e_{4k-5+3r},\\
 v_{4k-5+3r}&=a e_{4k-7+3r}+(a+1) e_{4k-6+3r}+e_{4k-5+3r}\\
\end{align*}
form a set of $6$ independent eigenvectors of $D$ associated to the eigenvalues $1,a,a+1$, and the following vectors are eigenvectors of $D$ associated to the eigenvalues $0,1$:
\begin{align*}
  v_{4k-1}&=\sum_{i=1}^{k-2}    (u_{4i-1} e_{4i-2}+ u_{4i-1} e_{4i-1})+ \sum_{j=0}^{1} (u_{4k-6+3j} e_{4k-7+3j}+u_{4k-6+3j} e_{4k-6+3j})\\&+    e_{4k-1}+   e_{4k}, \quad  \\
  v_{4k}&=\sum_{i=1}^{k-2}   ((a^2+a) u_{4i-1} e_{4i-2}+ u_{4i-1} e_{4i})\\& +\sum_{j=0}^{1}((a^2+a) u_{4k-6+3j} e_{4k-7+3j} +u_{4k-6+3j} e_{4k-5+3j})+ e_{4k}.
\end{align*}
Consequently, $\{v_1,\dots, v_{4k}\}$ is a basis of eigenvectors of $D$, and hence $D$ is diagonalizable. Observe in addition that, due to the structure of the matrix
$D$, it is readily verified that $v_{4k-1}$ and $v_{4k}$ are eigenvectors associated to $0,1$ within a $9\times 9$ matrix
$$D':=\left(
  \begin{array}{ccc}
    D_4 &0& T_1 \\
    0 & D_3& T_k \\
     0 & 0& D_2 \\
  \end{array}
\right).
$$
We thus have expressed $A$ as the sum of a square-zero matrix and a diagonalizable matrix with eigenvalues $0,1,a,a+1$, whence $B=cA$ can also be expressed as the sum of a square-zero matrix and a diagonalizable matrix with eigenvalues $0,c,ca, c(a+1)$.

\medskip

{\bf Case (b):} Suppose that $\text{Trace}(B)= 0$. Suppose also without loss of generality that $B$ has the form
$$
B=\left(\begin{array}{cccccc}
                                                                          0 & 0 & \cdots & 0 & 0 &u_0\\
                                                                          1 & 0 & \cdots & 0 & 0 &u_1\\
                                                                          0 & 1 & \cdots & 0 & 0 &u_2\\
                                                                          \vdots & \vdots & \ddots & \vdots & \vdots &\vdots\\
                                                                          0 & 0 & \cdots & 1 & 0 &u_{4k-2}\\
                                                                          0 & 0 & \cdots & 0 & 1 &0
                                                                        \end{array}\right).$$
We are now going to build a square-zero matrix $N$ such that $B=N+D$ and $D$ can be decomposed as the direct sum of one diagonalizable block of order $1$, $k-1$ diagonalizable blocks of order $4$ and one diagonalizable block of order $3$: To that end, let us consider
$$M_s:=\left(
   \begin{array}{ccc}
     0 & 0 &u_{4s-3}+(a^2+a) u_{4s}\\
     0 & 0 &u_{4s-2}+(a^2+a+1) u_{4s}\\
     0 & 0 &u_{4s-1}\\
     0 & 0 &0\\
   \end{array}
 \right),\hbox{ for $s=1,\dots, k-1$,}
$$
$$M_{k}:=\left(
   \begin{array}{ccc}
      0& 0 &u_{4k-3}+a^2+a\\
      0 & 0 &u_{4k-2}+a^2+a+1\\
      0 & 0 &0\\
   \end{array}
 \right)$$
and put
$$N:=\left(
      \begin{array}{cccccccc}
        0 & 0 & 0 & \cdots & 0 &0 &  0 \\
        Q_{4,1} & N_4 & 0 & \cdots &0 & 0  & M_1 \\
        0 & Q_{4,4} & N_4 & \cdots &0 & 0  & M_2 \\
        \vdots & \vdots & \ddots & \ddots  & \vdots & \vdots \\
        0 & 0 & 0 & \ddots &N_4 & 0 &  M_{k-2} \\
        0 & 0 & 0 & \cdots &Q_{4,4} & N_4 &  M_{k-1} \\
        0 & 0 & 0 & \cdots &0 & Q_{3,4} &  M_k \\
      \end{array}
    \right)\in M_{4k}(\mathbb{F}).
$$
It can be verified that $N_4 M_s=Q_{4,4} M_s$ for $s=1,\dots, k-1$, $Q_{3,4}M_{k-1}=0$, and $M_{k} M_{k}=0$, we conclude that $N$ is a square-zero matrix. Moreover, if we denote
$$
T_s:=\left(
   \begin{array}{ccc}
      0 & 0 &(a^2+a)u_{4s}\\
      0 & 0 &(a^2+a+1)u_{4s}\\
      0 & 0 &0\\
      0 &  &u_{4s}\\
   \end{array}
 \right),\hbox{$s=1\dots, k-1$, and}
$$
 $$T_0:=\left(
                 \begin{array}{ccc}
                   0 & 0 & u_0 \\
                 \end{array}
               \right),\quad T_{k}:=\left(
   \begin{array}{ccc}
       0 & 0 &a^2+a\\
       0 & 0 &a^2+a+1\\
       0 & 0 &0\\
   \end{array}
 \right),
 $$
we obtain that $B=N+D$, where
$$D:=\left(
         \begin{array}{ccccc}
           0 & 0 & \cdots  & 0 & T_0 \\
           0 & D_4 & \cdots  & 0 & T_1 \\
           \vdots & \vdots & \ddots  & \vdots & \vdots \\
           0 & 0 & \cdots  & D_4 & T_{k-1} \\
           0 & 0 & \cdots  & 0 & D_3 \\
         \end{array}
       \right).$$
Additionally, notice that $v_1=e_1$ is a eigenvector of $D$ associated to the eigenvalue $0$. Moreover, since $D_4$ is diagonalizable, for any $s=1,\dots, k-1$ the vectors
\begin{align*}
 v_{4s-2}&=(a^2+a)e_{4s-2}+(a^2+a+1)e_{4s-1}+e_{4s+1},\\
 v_{4s-1}&=(a^2+a)e_{4s-1}+e_{4s}+e_{4s+1},\\
 v_{4s}&=(a+1)e_{4s-1}+a e_{4s}+e_{4s+1},\\
 v_{4s+1}&=a e_{4s-1}+(a+1) e_{4s}+e_{4s+1}
\end{align*}
form a set of $4k-4$ independent eigenvectors of $D$ associated to $0,1,a,a+1$, and  the following vectors are eigenvectors of $D$ associated to the eigenvalues $1,a,a+1$:
\begin{align*}
  v_{4k-2}&=u_0 e_1+\sum_{i=1}^{k-1}  ( (a^2+a) u_{4i} e_{4i-2}+ (a^2+a+1) u_{4i} e_{4i-1}+ u_{4i} e_{4i+1})\\&+ (a^2+a) e_{4k-2}+  e_{4k-1}+ e_{4k},\\
  v_{4k-1}&=u_0 e_1+\sum_{i=1}^{k-1}  ( (a^2+a) u_{4i} e_{4i-2}+ a^2 u_{4i} e_{4i-1}+a u_{4i} e_{4i})\\&+ (a^2+a) e_{4k-2}+a^2  e_{4k-1}+ a e_{4k},\\
  v_{4k}&=u_0 e_1+\sum_{i=1}^{k-1}  ( (a^2+a) u_{4i} e_{4i-2}+ (a^2+1) u_{4i} e_{4i-1}+(a+1) u_{4i} e_{4i})\\&+ (a^2+a) e_{4k-2}+(a^2+1)  e_{4k-1}+ (a+1) e_{4k}.\\
\end{align*}
Therefore, $\{v_1,\dots, v_{4k}\}$ is a basis of eigenvectors of $D$, and hence $D$ is diagonalizable. Observe in addition that, due to the structure of the matrix
$D$, it is easily verified that $v_{4k-2}, v_{4k-1}$ and $v_{4k}$ are eigenvectors associated to the eigenvalues $1,a,a+1$ within a $7\times 7$ matrix
$$D':=\left(
  \begin{array}{cc}
    D_4 & T_1 \\
    0 & D_3\\
  \end{array}
\right).
$$
We thus have expressed $B$ as the sum of a square-zero matrix and a diagonalizable matrix with eigenvalues $0,1,a,a+1$, as promised.
\end{proof}

Our next consequence is the following.

\begin{corollary}\label{4k-corollary}
If $\FF=\FF_2$ and $B$ is a non-derogatory matrix of order $4k$ with $k\ge 2$ over $\FF$, then $B$ can be expressed as the sum of a square-zero matrix and a potent matrix $D$ with $D^4=D$.
\end{corollary}

\begin{proof} We distinguish two basic cases as follows.

\medskip

{\bf Case (a):} Suppose that $\text{Trace}(B)=1$. Suppose also, without loss of generality, that $B$ has the form
$$
B=\left(\begin{array}{cccccc}
                                                                          0 & 0 & \cdots & 0 & 0 &u_0\\
                                                                          1 & 0 & \cdots & 0 & 0 &u_1\\
                                                                          0 & 1 & \cdots & 0 & 0 &u_2\\
                                                                          \vdots & \vdots & \ddots & \vdots & \vdots &\vdots\\
                                                                          0 & 0 & \cdots & 1 & 0 &u_{4k-2}\\
                                                                          0 & 0 & \cdots & 0 & 1 &1
                                                                        \end{array}\right)\in M_{4k}(\FF_2)\subseteq M_{4k}(\FF_4).$$
Adapting Proposition \ref{4k} Case (a) for $\FF=\FF_4$, if we consider $a\in \FF_4$ such that $a^2+a+1=0$, then one finds that $B=N+D$, where
$$N=\left(
      \begin{array}{cccccccc}
        N_4  & 0 &\cdots & 0& 0 & 0 & 0&  M_1 \\
        Q_{4,4}  &N_4 & \cdots & 0 & 0 & 0& 0  & M_2 \\
        \vdots &\ddots & \ddots & \vdots& \vdots & \vdots& 0  & \vdots \\
        0 &0 & \ddots & N_4 & 0& 0 & 0&  M_{k-3} \\
        0 &0 & \cdots & Q_{4,4} & N_4& 0& 0 &  M_{k-2} \\
        0 &0 & \cdots & 0 & Q_{3,4}& N_3& 0 &  M_{k-1} \\
        0 &0 & \cdots & 0 & 0& Q_{3,3} &N_3&   M_k \\
        0 &0 & \cdots & 0 &0 & 0& Q_{2,3} &  M_{k+1} \\
      \end{array}
    \right)
$$
\noindent and
$$
D=\left(
         \begin{array}{cccccc}
           D_4 & \cdots & 0 & 0 & 0 & T_1 \\
           \vdots & \ddots & \vdots & \vdots & \vdots & \vdots \\
           0 & \cdots & D_4 & 0 & 0 & T_{k-2} \\
           0 & \cdots & 0 & D_3 & 0 & T_{k-1} \\
           0 & \cdots & 0 & 0 & D_3 & T_k \\
           0 & \cdots & 0 & 0 & 0 & D_2 \\
         \end{array}
       \right).
$$
Notice that $N\in M_{4k}(\FF_2)$, because
$$
N_3=\left(
           \begin{array}{ccc}
             0 & 0 & a^2+a \\
             0 & 0 & a^2+a+1 \\
             0 & 0 & 0 \\
           \end{array}
         \right)=\left(
           \begin{array}{ccc}
             0 & 0 & 1 \\
             0 & 0 & 0 \\
             0 & 0 & 0 \\
           \end{array}
         \right)\in M_3(\FF_2),
$$
$$
N_{4}=\left(
           \begin{array}{cccc}
             0 & 0 & 0 & 0 \\
             0 & 0 & 0 & a^2+a \\
             0 & 0 & 0 & a^2+a+1 \\
             0 & 0 & 0 & 0 \\
           \end{array}
         \right)=\left(
           \begin{array}{cccc}
             0 & 0 & 0 & 0 \\
             0 & 0 & 0 & 1 \\
             0 & 0 & 0 & 0 \\
             0 & 0 & 0 & 0 \\
           \end{array}
         \right)\in M_4(\FF_2)
$$
and, by construction, all $M_i$ and $Q_{i,j}$ with $i,j=1,\dots, k+1$ are matrices over $\FF_2$; whence $D=B+N\in M_{4k}(\FF_2)$. Moreover, we also have that $N^2=0$ and $D^4=D$, because $D$ is diagonalizable over $\FF_4$ with eigenvalues $0,1,a, a+1$.

\medskip

{\bf Case (b):} Suppose that $\text{Trace}(B)=0$. Suppose also, without loss of generality, that $B$ has the form
$$
B=\left(\begin{array}{cccccc}
                                                                          0 & 0 & \cdots & 0 & 0 &u_0\\
                                                                          1 & 0 & \cdots & 0 & 0 &u_1\\
                                                                          0 & 1 & \cdots & 0 & 0 &u_2\\
                                                                          \vdots & \vdots & \ddots & \vdots & \vdots &\vdots\\
                                                                          0 & 0 & \cdots & 1 & 0 &u_{4k-2}\\
                                                                          0 & 0 & \cdots & 0 & 1 &0
                                                                        \end{array}\right)\in M_{4k}(\FF_2)\subseteq M_{4k}(\FF_4).$$
Adapting Proposition \ref{4k} Case (b) for $\FF=\FF_4$, if we consider $a\in \FF_4$ such that $a^2+a+1=0$, then one detects that $B=N+D$, where
$$N=\left(
      \begin{array}{cccccccc}
        0 & 0 & 0 & \cdots & 0 &0 &  0 \\
        Q_{4,1} & N_4 & 0 & \cdots &0 & 0  & M_1 \\
        0 & Q_{4,4} & N_4 & \cdots &0 & 0  & M_2 \\
        \vdots & \vdots & \ddots & \ddots  & \vdots & \vdots \\
        0 & 0 & 0 & \ddots &N_4 & 0 &  M_{k-2} \\
        0 & 0 & 0 & \cdots &Q_{4,4} & N_4 &  M_{k-1} \\
        0 & 0 & 0 & \cdots &0 & Q_{3,4} &  M_k \\
      \end{array}
    \right)
$$
\noindent and
$$
D=\left(
         \begin{array}{ccccc}
           0 & 0 & \cdots  & 0 & T_0 \\
           0 & D_4 & \cdots  & 0 & T_1 \\
           \vdots & \vdots & \ddots  & \vdots & \vdots \\
           0 & 0 & \cdots  & D_4 & T_{k-1} \\
           0 & 0 & \cdots  & 0 & D_3 \\
         \end{array}
       \right).
$$
Notice that $N\in M_{4k}(\FF_2)$, because
$$
N_{4}=\left(
           \begin{array}{cccc}
             0 & 0 & 0 & 0 \\
             0 & 0 & 0 & a^2+a \\
             0 & 0 & 0 & a^2+a+1 \\
             0 & 0 & 0 & 0 \\
           \end{array}
         \right)=\left(
           \begin{array}{cccc}
             0 & 0 & 0 & 0 \\
             0 & 0 & 0 & 1 \\
             0 & 0 & 0 & 0 \\
             0 & 0 & 0 & 0 \\
           \end{array}
         \right)\in M_4(\FF_2),
$$
and, for $s=1,\dots, k-1$,
$$M_s=\left(
   \begin{array}{ccc}
     0 & 0 &u_{4s-3}+(a^2+a) u_{4s}\\
     0 & 0 &u_{4s-2}+(a^2+a+1) u_{4s}\\
     0 & 0 &u_{4s-1}\\
     0 & 0 &0\\
   \end{array}
 \right)=\left(
   \begin{array}{ccc}
     0 & 0 &u_{4s-3}+ u_{4s}\\
     0 & 0 &u_{4s-2}\\
     0 & 0 &u_{4s-1}\\
     0 & 0 &0\\
   \end{array}
 \right)\in M_{4,3}(\FF_2),
$$
$$M_{k}=\left(
   \begin{array}{ccc}
      0& 0 &u_{4k-3}+a^2+a\\
      0 & 0 &u_{4k-2}+a^2+a+1\\
      0 & 0 &0\\
   \end{array}
 \right)=\left(
   \begin{array}{ccc}
      0& 0 &u_{4k-3}+1\\
      0 & 0 &u_{4k-2}\\
      0 & 0 &0\\
   \end{array}
 \right)\in M_3(\FF_2)$$
and, by construction, all $Q_{i,j}$ with $i,j=1,\dots, k+1$ are matrices over $\FF_2$; whence $D=B+N\in M_{4k}(\FF_2)$. Moreover, we also have  that $N^2=0$ and $D^4=D$, because $D$ is diagonalizable over $\FF_4$ with eigenvalues $0,1,a, a+1$.
\end{proof}

In the following, we deal with non-derogative matrices of order $4k+1$, where $k\ge 1$.

\begin{proposition}\label{4k+1}
Let $\FF$ be a field of characteristic 2 with $\#\FF\ge4$, let $B$ be a non-derogative matrix over $\FF$ of order $4k+1$ with $k\ge 1$, and let $\text{Trace}(B)=c$. Then, for every $a\in \FF$, $a\ne 0,1$, there exists a square-zero matrix $N\in {M}_{4k+1}(\FF)$ such that $B=N+D$, where $D$ is diagonalizable with eigenvalues $c,c+1,c+a$, $c+a+1$.
\end{proposition}

\begin{proof} Since $4k+1$ is odd, we can freely replace $B$ by $A=B+(c+1)\Id$, which is again a non-derogative matrix which has trace equal to 1. We can assume without loss of generality that
$$A=\left(\begin{array}{cccccc}
                                                                          0 & 0 & \cdots & 0 & 0 &u_0\\
                                                                          1 & 0 & \cdots & 0 & 0 &u_1\\
                                                                          0 & 1 & \cdots & 0 & 0 &u_2\\
                                                                          \vdots & \vdots & \ddots & \vdots & \vdots &\vdots\\
                                                                          0 & 0 & \cdots & 1 & 0 &u_{4k-1}\\
                                                                          0 & 0 & \cdots & 0 & 1 &1
                                                                        \end{array}\right).$$
Let us show now that we can build a square-zero matrix $N$ such that $A=N+D$, where $D$ is the direct sum of $k-1$ diagonalizable matrices of order $4$, one diagonalizable matrix of order $3$ and one diagonalizable matrix of order $2$. To that fact, let us consider
$$M_s:=\left(
   \begin{array}{cc}
      0 &-u_{4s-4}\\
      0 &-u_{4s-3}\\
      0 &u_{4s-1}-u_{4s-2}\\
      0 &0\\
   \end{array}
 \right),\hbox{ for $s=1,\dots, k-1$,}$$
 $$M_{k}:=\left(
   \begin{array}{ccc}
      0& 0 &-u_{4k-4}\\
      0 & 0 &u_{4k-2}-u_{4k-3}\\
      0 & 0 &0\\
   \end{array}
 \right),\ M_{k+1}:=\left(
   \begin{array}{cc}
      0 &-u_{4k-1}\\
      0 &0\\
   \end{array}
 \right),$$
and put
$$N:=\left(
      \begin{array}{ccccccc}
        N_4 & 0 & \cdots &  0 & 0& 0 & M_1 \\
        Q_{4,4} & N_4 & \cdots  & 0& 0 & 0 & M_2 \\
        \vdots & \ddots & \ddots  & \vdots &\vdots & \vdots & \vdots \\
        0 & 0 & \ddots   & N_4 & 0& 0 & M_{k-1} \\
        0 & 0 & \ddots   & Q_{4,4} & N_4& 0 & M_{k-1} \\
        0 & 0 & \cdots  &0&  Q_{3,4} & N_3 & M_{k} \\
        0 & 0 & \cdots  & 0& 0 &Q_{2,3} & M_{k+1} \\
      \end{array}
    \right)\in M_{4k+1}(\mathbb{F}).$$
Since one sees that $N_4 M_s=Q_{4,4} M_s=0$ for $s=1,\dots, k-1$, $Q_{3,4}M_{k-1}=N_3M_k=Q_{2,3}M_k=0$, and $M_{k+1} M_{k+1}=0$, we infer that $N$ is a square-zero matrix. Moreover, if we consider
$$T_s:=\left(
   \begin{array}{cc}
      0 &0\\
      0 &0\\
      0 &u_{4s-1}\\
      0 &u_{4s-1}\\
   \end{array}
 \right),\hbox{ for $s=1\dots, k-1$, and} ~ T_{k}:=\left(
   \begin{array}{cc}
       0 &0\\
       0 &u_{4k-2}\\
       0 &u_{4k-2}\\
   \end{array}
 \right),
$$
we perceive that $A=N+D$, where
$$D:=\left(
         \begin{array}{ccccc}
           D_4 & \cdots & 0 & 0 & T_1 \\
           \vdots & \ddots & \vdots & \vdots & \vdots \\
           0 & \cdots & D_4 & 0 & T_{k-1} \\
           0 & \cdots & 0 & D_3 & T_k \\
           0 & \cdots & 0 & 0 & D_2 \\
         \end{array}
       \right).$$
Since $D_4$ is diagonalizable, we derive that, for any $s=1\dots, k-1$, the vectors
\begin{align*}
 v_{4s-3}&=(a^2+a)e_{4s-3}+(a^2+a+1)e_{4s-2}+e_{4s},\\
 v_{4s-2}&=(a^2+a)e_{4s-2}+e_{4s-1}+e_{4s},\\
 v_{4s-1}&=(a+1)e_{4s-2}+a e_{4s-1}+e_{4s},\\
 v_{4s}&=a e_{4s-2}+(a+1) e_{4s-1}+e_{4s}
\end{align*}
form a set of $4k-4$ independent eigenvectors of $D$ associated to the eigenvalues $0,1,a,a+1$. Moreover, since $D_{3}$ is diagonalizable,
\begin{align*}
 v_{4k-3}&=(a^2+a)e_{4s-3}+e_{4s-2}+e_{4s-1},\\
 v_{4k-2}&=(a+1)e_{4s-3}+a e_{4s-2}+e_{4s-1},\\
 v_{4k-1}&=a e_{4s-3}+(a+1)e_{4s-2}+e_{4s-1}
\end{align*}
are eigenvectors of $D$ associated to $1,a,a+1$. Likewise, the following vectors are eigenvectors of $D$ associated to $0,1$:
 \begin{align*}
 v_{4k}&=\sum_{i=1}^{k-1}  ( u_{4i-1} e_{4i-2}+ u_{4i-1} e_{4i-1})+ u_{4k-2} e_{4k-3}+ u_{4k-2} e_{4k-2}+ e_{4k}+  e_{4k+1},\\
 v_{4k+1}&=\sum_{i=1}^{k-1} ((a^2+a) u_{4i-1} e_{4i-2}+ u_{4i-1} e_{4i})+ (a^2+a) u_{4k-2} e_{4k-3}+ u_{4k-2} e_{4k-1}\\&+   e_{4k+1}.\\
\end{align*}
That is why $\{v_1,\dots, v_{4k+1}\}$ form a basis of eigenvectors of $D$, and hence it is diagonalizable. Observe additionally that, due to the structure of the matrix $D$, it is plainly observed that $v_{4k}, v_{4k+1}$ are eigenvectors associated to $0,1$ within a $9\times 9$ matrix
$$D':=\left(
  \begin{array}{ccc}
    D_4&0 & T_1 \\
    0 & D_3&T_2\\
    0&0& D_2 \\
  \end{array}
\right).
$$
We thus have expressed $A$ as the sum of a square-zero matrix and a diagonalizable matrix with eigenvalues $0,1,a,a+1$, so that $B=A+(c+1)\Id$ can also be expressed as the sum of a square-zero matrix and a diagonalizable matrix with eigenvalues $c+1,c,a+c+1,a+c$, as desired.
\end{proof}

As a new consequence, we find:

\begin{corollary}\label{4k+1-corollary}
If $\FF=\FF_2$ and $B$ is a non-derogatory matrix of order $4k+1$ with $k\ge 1$ over $\FF$, then $B$ can be expressed as the sum of a square-zero matrix $N$ and a potent matrix $D$ with $D^4=D$.
\end{corollary}

\begin{proof}
The claim follows from Proposition \ref{4k+1} by considering $B$ as a matrix over $\FF_4$ and taking $a\in \FF_4$ such that $a^2+a+1=0$; notice also that $N$ and $D$ are matrices over $\FF_2$ by imitating the same arguments as in the proof of Case (a) of Corollary \ref{4k-corollary}. Moreover, one knows that $D^4=D$, because it is a diagonalizable matrix over $\FF_4$ with eigenvalues $0,1,a,a+1$.
\end{proof}

Now, we study non-derogative matrices of order $4k+2$, where $k\ge 1$.

\begin{proposition}\label{4k+2}
Let $\FF$ be a field of characteristic 2 with $\#\FF\ge4$, and let $B$ be a non-derogative matrix over $\FF$ of order $4k+2$ with $k\ge 1$. Then, for every $a\in \FF$, $a\ne 0,1$:
\begin{itemize}
\item[(a)] If $\text{Trace}(B)=c\ne 0$, there exists a square-zero matrix $N\in {M}_{4k+2}(\FF)$ such that $B=N+D$, where $D$ is diagonalizable with eigenvalues $0, c, ca,c(a+1)$.
\item[(b)] If $\text{Trace}(B)=0$ for $b\in \mathbb{F}$ such that $b^2 =w_{4k}+a^2+a+1$, where $w_{4k}$ is the coefficient of degree $4k$ of the characteristic polynomial of $B$, there exists a square-zero matrix $N\in {M}_{4k+2}(\FF)$ such that $B=N+D$, where $D$ is diagonalizable with eigenvalues $b,b+1,b+a$ and $b+a+1$.
\end{itemize}
\end{proposition}

\begin{proof}
{\bf Case (a):} Suppose that $\text{Trace}(B)=c\ne 0$. Let us consider the non-derogative matrix $A=c^{-1} B$, which has trace equal to one. Suppose also, without loss of generality, that $A$ has the form
$$A=\left(\begin{array}{cccccc}
                                                                          0 & 0 & \cdots & 0 & 0 &u_0\\
                                                                          1 & 0 & \cdots & 0 & 0 &u_1\\
                                                                          0 & 1 & \cdots & 0 & 0 &u_2\\
                                                                          \vdots & \vdots & \ddots & \vdots & \vdots &\vdots\\
                                                                          0 & 0 & \cdots & 1 & 0 &u_{4k}\\
                                                                          0 & 0 & \cdots & 0 & 1 &1
                                                                        \end{array}\right).$$

Let us show now that we can build a square-zero matrix $N$ such that $A=N+D$, where $D$ is direct sum of $k$ diagonalizable matrices of order $4$ and one diagonalizable matrix of order $2$. To that end, let us consider
$$M_s:=\left(
   \begin{array}{cc}
      0 &-u_{4s-4}\\
      0 &  u_{4s-1}-u_{4s-3}\\
      0 & -u_{4s-2}\\
      0 &0\\
   \end{array}
 \right),\hbox{ for $s=1,\dots, k$, and} ~ M_{k+1}:=\left(
   \begin{array}{cc}
      0 &-u_{4k}\\
      0 &0\\
   \end{array}
 \right),$$
and put
$$N:=\left(
      \begin{array}{cccccc}
        N_4 & 0 & 0 &  0 & 0 & M_1 \\
        Q_{4,4} & N_4 & 0  & 0 & 0 & M_2 \\
        0 & Q_{4,4} & N_4  & 0 & 0 & M_3 \\
        \vdots & \vdots &  \ddots & \ddots & \vdots & \vdots \\
        0 & 0 & 0  & Q_{4,4} & N_4 & M_{k} \\
        0 & 0 & 0  & 0 & Q_{2,4} & M_{k+1} \\
      \end{array}
    \right).$$
Since  $N_4 M_s=Q_{4,4} M_s=0$, for $s=1,\dots, k-1$, $N_4M_k=0$, $Q_{2,4} M_k=0$, and $M_{k+1} M_{k+1}=0$, we arrive at the fact that $N$ is a square-zero matrix. Moreover, if we consider
$$T_s:=\left(
    \begin{array}{ccc}
       0 & 0 & 0 \\
       0 & 0 &  u_{4s-1} \\
       0 & 0 & 0   \\
       0 & 0 & u_{4s-1} \\
    \end{array}
  \right),\hbox{ for } s=1\dots, k-1,$$
we obtain $A=N+D$, where
$$D:=\left(
  \begin{array}{ccccc}
    D_4 & 0 & \cdots  & 0 & T_1 \\
    0 & D_4 & \cdots   & 0 & T_2 \\
    \vdots & \vdots & \ddots & \vdots &  \vdots \\
    0 & 0 & \cdots   & D_4 & T_k \\
    0 & 0 & \cdots   & 0 & D_2 \\
  \end{array}
\right).$$
Since $D_4$ is diagonalizable, we reach that, for $s=1\dots, k$ the vectors
\begin{align*}
 v_{4s-3}&=(a^2+a)e_{4s-3}+(a^2+a+1)e_{4s-2}+e_{4s},\\
 v_{4s-2}&=(a^2+a)e_{4s-2}+e_{4s-1}+e_{4s},\\
 v_{4s-1}&=(a+1)e_{4s-2}+a e_{4s-1}+e_{4s},\\
 v_{4s}&=a e_{4s-2}+(a+1) e_{4s-1}+e_{4s}
\end{align*}
form a set of $4k$ independent eigenvectors of $D$ associated to the eigenvalues $0,1,a,a+1$, and the following vectors are eigenvectors of $D$ associated to the eigenvalues $0,1$:
\begin{align*}
  v_{4k+1}&=\sum_{i=1}^{k} (  u_{4i-1} e_{4i-3}+ u_{4i-1} e_{4i-1})+ e_{4k+1}+  e_{4k+2},\\
  v_{4k+2}&=\sum_{i=1}^{k}  ( u_{4i-1} e_{4i-2}+ u_{4i-1} e_{4i-1})+  e_{4k+2}.
\end{align*}

Consequently, we extract that $\{v_1,\dots, v_{2k+2}\}$ is a basis of eigenvectors of $D$, and hence it is diagonalizable. Observe in addition that, due to the structure of the matrix
$D$, one can check that $v_{4k+1}, v_{4k+2}$ are eigenvectors associated to the eigenvalues $0,1$ within the $6\times 6$ matrix
$$D':=\left(
  \begin{array}{cc}
    D_4 & T_1 \\
    0 & D_2 \\
  \end{array}
\right).
$$

We thus have expressed $A$ as the sum of a square-zero matrix and a diagonalizable matrix with eigenvalues $0,1,a,a+1$, so that $B=cA$ can also be expressed as the sum of a square-zero matrix and a diagonalizable matrix with eigenvalues $0,c,ca, c(a+1)$.

\medskip

{\bf Case (b):} Suppose that $\text{Trace}(B)= 0$ and let $$p_B(x)=x^{4k+2}+w_{4k}x^{4k}+\dots$$ be the characteristic polynomial of $B$. Since the map $\phi: \FF\to \FF$ defined by $\phi(x)=x^2$ is a one-to-one $\mathbb{F}_2$-linear map, it is surjective and so we can choose $b\in \FF$ such that $b^2 =w_{4k}+a^2+a+1$. Let us work with the non-derogative matrix $A=B+b\Id$ whose characteristic polynomial is of the form $$p_{A}(x)=x^{4k+2}+ u_{4k+1} \, x^{4k+1}+u_{4k}\, x^{4k}+\cdots ,$$
$$\hbox{where} ~ u_{4k+1}=0 ~ \hbox{and} ~ u_{4k}=w_{4k}+b^2=a^2+a+1.$$
We also can suppose without loss of generality that $A$ is of the form
 $$A=\left(\begin{array}{cccccc}
                                                                          0 & 0 & \cdots & 0 & 0 &u_0\\
                                                                          1 & 0 & \cdots & 0 & 0 &u_1\\
                                                                          \vdots & \ddots & \ddots & \vdots & \vdots &\vdots\\
                                                                          0 & 0 & \ddots & 0 & 0 &u_{4k-1}\\
                                                                          0 & 0 & \cdots & 1 & 0 &a^2+a+1\\
                                                                          0 & 0 & \cdots & 0 & 1 &0
                                                                        \end{array}\right).$$

Let us prove now that we can build a square-zero matrix $N$ such that $A=N+D$, where $D$ is the direct sum of $k-1$ diagonalizable matrices of order $4$ and two diagonalizable matrices of order $3$. To that end, let us consider
$$M_s:=\left(
   \begin{array}{ccc}
     0 & 0 &-u_{4s-4}\\
     0 & 0 &-u_{4s-3}\\
     0 & u_{4s-1} &-u_{4s-2}\\
     0 & 0 &0\\
   \end{array}
 \right),\hbox{ for $s=1,\dots, k-1$,}
$$
$$M_k:=\left(
   \begin{array}{ccc}
     0 & 0 &-u_{4k-4}\\
     0 & u_{4k-2} &-u_{4k-3}\\
     0 & 0 &0\\
   \end{array}
 \right),\ M_{k+1}:=\left(
   \begin{array}{ccc}
     0 & 0 &-u_{4k-1}+a^2+a\\
     0 & 0 &0\\
     0 & 0 &0\\
   \end{array}
 \right),$$
and put
$$N:=\left(
      \begin{array}{cccccccc}
        N_4 & 0 & 0&\cdots & 0 & 0 & 0 & M_1 \\
        Q_{4,4} & N_4 &0& \cdots & 0 & 0 & 0 & M_2 \\
        0 & Q_{4,4} & N_4&\ddots & 0 & 0 & 0 & M_3 \\
        \vdots & \vdots& \ddots & \ddots & \vdots &\vdots & \vdots & \vdots \\
        0 & 0& 0 & \ddots & N_4 & 0 & 0 & M_{k-2} \\
        0 & 0& 0 & \cdots & Q_{4,4} & N_4 & 0 & M_{k-1} \\
        0 & 0& 0 & \cdots & 0 & Q_{3,4} & N_3 & M_{k} \\
        0 & 0& 0 & \cdots & 0 & 0 & Q_{2,3} & M_{k+1} \\
      \end{array}
    \right).$$
Since  $N_4 M_s=Q_{4,4} M_s=0$ for $s=1,\dots, k-2$, $N_4M_k=Q_{3,4}M_k=0$, $N_3M_k=Q_{2,3}M_k=0$, and $M_{k+1} M_{k+1}=0$, we detect that $N$ is a square-zero matrix. Moreover, if we consider
$$T_s:=\left(
    \begin{array}{ccc}
       0 & 0 & 0 \\
       0 & 0 & 0 \\
       0 & u_{4s-1} &  0\\
       0 & 0 & u_{4s-1} \\
    \end{array}
  \right),\hbox{ for $s=1\dots, k-1$,}\  T_k:=\left(
             \begin{array}{ccc}
               0 & 0 & 0 \\
               0 & u_{4k-4} & 0 \\
               0 & 0 & u_{4k-4} \\
             \end{array}
           \right),$$
we discover that
$$D:=A+N=\left(
  \begin{array}{cccccc}
    D_4 & 0 & \cdots & 0 & 0 & T_1 \\
    0 & D_4 & \cdots  & 0 & 0 & T_2 \\
    \vdots & \vdots & \ddots & \vdots & \vdots& \vdots \\
    0&0 & \cdots  & D_4 &0&T_{k-1} \\
    0 & 0 & \cdots  & 0 & D_3 & T_k \\
    0 & 0 & \cdots  & 0 & 0 & D_3 \\
  \end{array}
\right)\in M_{4k+2}(\mathbb{F}).$$
Furthermore, since $D_4$ is diagonalizable, we establish that, for $s=1\dots, k-1$ the vectors
\begin{align*}
 v_{4s-3}&=(a^2+a)e_{4s-3}+(a^2+a+1)e_{4s-2}+e_{4s},\\
 v_{4s-2}&=(a^2+a)e_{4s-2}+e_{4s-1}+e_{4s},\\
 v_{4s-1}&=(a+1)e_{4s-2}+a e_{4s-1}+e_{4s},\\
 v_{4s}&=a e_{4s-2}+(a+1) e_{4s-1}+e_{4s}
\end{align*}
form a set of $4k-4$ independent eigenvectors of $D$ associated to the eigenvalues $0,1,a,a+1$. Moreover, since $D_3$ is diagonalizable,
\begin{align*}
 v_{4k-3}&=(a^2+a)e_{4k-3}+e_{4k-2}+e_{4k-1},\\
 v_{4k-2}&=(a+1)e_{4k-3}+a e_{4k-2}+e_{4k-1},\\
 v_{4k-1}&=a e_{4k-3}+(a+1)e_{4k-2}+e_{4k-1}
\end{align*}
are three independent vectors which are eigenvectors of $D$ associated to the eigenvalues $1,a,a+1$, and the following vectors are eigenvectors of $D$ associated to the eigenvalues $1,a,a+1$:
\begin{align*}
  v_{4k}&=\sum_{i=1}^{k-1} ( (a^2+a)u_{4i-1} e_{4i-2}+u_{4i-1} e_{4i})+(a^2+a)u_{4k-2} e_{4k-3}+ u_{4k-2} e_{4k-1}   \\&+(a^2+a)e_{4k}+e_{4k+1}+ e_{4k+2},\\
  v_{4k+1}&=\sum_{i=1}^{k-1} ( (a+1)u_{4i-1}e_{4i-2}+ u_{4i-1} e_{4i})+(a+1)u_{4k-2} e_{4k-3} +u_{4k-2} e_{4k-1}\\&  +(a^2+a)e_{4k}+a^2e_{4k+1}+a e_{4k+2},\\
  v_{4k+2}&=\sum_{i=1}^{k-1} ( a u_{4i-1}e_{4i-2}+ u_{4i-1} e_{4i})+a u_{4k-2} e_{4k-3} +u_{4k-2} e_{4k-1}\\&+(a^2+a)e_{4k}+(a^2+1) e_{4k+1}+(a+1) e_{4k+2}.
\end{align*}
Therefore, $\{v_1,\dots, v_{2k+2}\}$ is a basis of eigenvectors of $D$, and hence $D$ is diagonalizable. Observe additionally that, due to the structure of the matrix
$D$, one sees that $v_{4k}, v_{4k+1}$ and $ v_{4k+2}$ are eigenvectors associated to the eigenvalues $1,a,a+1$ within the $7\times 7$ matrix
$$D':=\left(
  \begin{array}{cc}
    D_4 & T_1 \\
    0 & D_3 \\
  \end{array}
\right).
$$

We thus have shown that $A$ can be expressed as the sum of a square-zero matrix and a diagonalizable matrix with eigenvalues $0,1,a, a+1$. Hence, $B=A+b \Id$ can also be expressed as the sum of a square-zero matrix and a diagonalizable matrix with eigenvalues $b,1+b,a+b, a+1+b$, where $b^2 =w_{4k}+a^2+a+1$ and $w_{4k}$ is the coefficient of degree $4k$ of the characteristic polynomial of $B$, as wanted.
\end{proof}

Our common consequence is the following.

\begin{corollary}\label{4k+2-corollary}
If $\FF=\FF_2$ and $B$ is a non-derogatory matrix of order $4k+2$ with $k\ge 1$ over $\FF$, then $B$ can be expressed as the sum of a square-zero matrix $N$ and a potent matrix $D$ with $D^4=D$.
\end{corollary}

\begin{proof} We can consider $B$ as a matrix over $\FF_4$, and take $a\in \FF_4$ such that $a^2+a+1=0$.

If $\text{Trace}(B)=1$, then Proposition \ref{4k+2} Case (a) shows that $B=N+D$, where $N^2=0$ and $D$ is diagonalizable over $\FF_4$ with eigenvalues $0, 1, a, a+1$, so $D^4=D$. Notice that both $N$ and $D$ are matrices over $\FF_2$ via the same arguments as in the proof of Case (a) of Corollary \ref{4k-corollary}.

If now $\text{Trace}(B)=0$, then with the aid of Proposition \ref{4k+2} Case (b) we can take $b =w_{4k}$, where $w_{4k}$ is the coefficient of degree $4k$ of the characteristic polynomial of $B$; then, $B=N+D$ where $N^2=0$ and $D$ is diagonalizable over $\FF_4$ with eigenvalues $b,b+1,b+a,b+a+1$ and hence $D^4=D$. Notice that both $N$ and $D$ are matrices over $\FF_2$ by the same arguments as in the proof of Case (a) of Corollary \ref{4k-corollary}.
\end{proof}

Finally, let us deal with non-derogative matrices of $4k+3$, where $k\ge 1$.

\begin{proposition}\label{4k+3}
Let $\FF$ be a field of characteristic 2 with $\#\FF\ge4$, let $B$ be a non-derogative matrix over $\FF$ of order $4k+3$ with $k\ge 1$, and let $c=\text{Trace}(B)$. Then, for every $a\in \FF$, $a\ne 0,1$, there exists a square-zero matrix $N$ such that $B=N+D$, where $D$ is diagonalizable with eigenvalues $c,c+1,c+a,c+a+1$.
\end{proposition}

\begin{proof}
Since $4k+3$ is odd, without worrying we can  replace $B$ by $A=B+c\Id$, which has zero trace. Since the claim follows up to similarity, we may assume without loss of generality that $$A=\left(\begin{array}{cccccc}
                                                                          0 & 0 & \cdots & 0 & 0 &u_0\\
                                                                          1 & 0 & \cdots & 0 & 0 &u_1\\
                                                                          0 & 1 & \cdots & 0 & 0 &u_2\\
                                                                          \vdots & \vdots & \ddots & \vdots & \vdots &\vdots\\
                                                                          0 & 0 & \cdots & 1 & 0 &u_{4k+1}\\
                                                                          0 & 0 & \cdots & 0 & 1 &0
                                                                        \end{array}\right).$$

Let us now build a square-zero matrix $N$ such that $A=N+D$ and $D$ is the direct sum of $k$ diagonalizable blocks of order 4 and one diagonalizable block of order 3. To that direction, let us consider
$$M_s:=\left(
 \begin{array}{ccc}
  0 & 0 & (a^2+a)u_{4s-1}-u_{4s-4} \\
  0 & 0 & (a^2+a+1)u_{4s-1}-u_{4s-3} \\
   0 & 0 & -u_{4s-2} \\
    0 & 0 & 0 \\
   \end{array}
   \right),\hbox{ for $s=1,\dots, k$,}
$$
$$M_{k+1}:=\left(
                    \begin{array}{ccc}
                      0 & 0 & a^2+a-u_{4k} \\
                      0 & 0 & a^2+a+1-u_{4k+1} \\
                      0 & 0 & 0 \\
                    \end{array}
                  \right),$$
and put
$$N:=\left(
      \begin{array}{ccccccc}
        N_4 & 0 & 0 & \cdots & 0 & 0 & M_1 \\
        Q_{4,4} & N_4 & 0 & \cdots & 0 & 0 & M_2 \\
        0 & Q_{4,4} & N_4 & \cdots & 0 & 0 & M_3 \\
        \vdots & \vdots & \ddots & \ddots & \vdots & \vdots & \vdots \\
        0 & 0 & 0 & \ddots & N_4 & 0 & \vdots \\
        0 & 0 & 0 & \cdots & Q_{4,4} & N_4 & M_{k} \\
        0 & 0 & 0 & \cdots & 0 & Q_{3,4} & M_{k+1} \\
      \end{array}
    \right).$$
Since  $N_4 M_s=Q_{4,4} M_s=0$ for $s=1,\dots, k-1$, $N_4M_k=Q_{3,4}M_k=0$, and $M_{k+1} M_{k+1}=0$, we conclude that $N$ is a square-zero matrix. Moreover, if we consider
$$
T_s:=\left(
 \begin{array}{ccc}
  0 & 0 & (a^2+a)u_{4s-1} \\
  0 & 0 & (a^2+a+1)u_{4s-1}  \\
  0 & 0 & 0 \\
  0 & 0 & u_{4s-1} \\
  \end{array}
  \right), \hbox{ for $s=1\dots, k$,}
$$ we are able to write $A=N+D$, where
$$D:=\left(
          \begin{array}{ccccc}
            D_4 & 0 & 0 & 0 & T_1 \\
            0 & D_4 & 0 & 0 & T_2 \\
            0 & 0 & \ddots & 0 & \vdots \\
            0 & 0 & 0 & D_4 & T_{k} \\
            0 & 0 & 0 & 0 & D_3 \\
          \end{array}
        \right)\in M_{4k+3}(\mathbb{F}).$$
Since $D_4$ is diagonalizable, we detect that, for $s=1\dots, k$ the vectors
\begin{align*}
 v_{4s-3}&=(a^2+a)e_{4s-3}+(a^2+a+1)e_{4s-2}+e_{4s},\\
 v_{4s-2}&=(a^2+a)e_{4s-2}+e_{4s-1}+e_{4s},\\
 v_{4s-1}&=(a+1)e_{4s-2}+a e_{4s-1}+e_{4s},\\
 v_{4s}&=a e_{4s-2}+(a+1) e_{4s-1}+e_{4s}
\end{align*}
form a set of $4k$ independent eigenvectors of eigenvalues $0,1,a,a+1$. Moreover, the next vectors are eigenvectors of eigenvalues $1,a,a+1$:
\begin{align*}
  v_{4k+1}&=\sum_{i=1}^k ( (a^2+a)u_{4i-1} e_{4i-3}+(a^2+a+1)u_{4i-1}e_{4i-2}+ u_{4i-1} e_{4i})\\&+(a^2+a)e_{4k+1}+e_{4k+2}+ e_{4k+3},\\
  v_{4k+2}&=\sum_{i=1}^k ((a+1)u_{4i-1} e_{4i-3}+a u_{4i-1}e_{4i-2}+ u_{4i-3} e_{4i-1})\\&+(a+1)e_{4k+1}+a e_{4k+2}+ e_{4k+3},\\
  v_{4k+3}&=\sum_{i=1}^k  (a u_{4i-1} e_{4i-3}+(a+1)u_{4i-1}e_{4i-2}+ u_{4i-3} e_{4i-1})+a e_{4k+1}+(a+1) e_{4k+2}\\&+ e_{4k+3}.
\end{align*}
Moreover, $B=\{v_1,\dots, v_{4k+3}\}$ is a basis of eigenvectors, and hence $D$ is diagonalizable. Observe in addition that, due to the structure of the matrix
$D$, one observes that $v_{4k+1}, v_{4k+2}$ and $v_{4k+3}$ are eigenvectors of eigenvalues $1,a,a+1$ within the $7\times 7$ matrix
$$D':=\left(
  \begin{array}{cc}
    D_4 & T_1 \\
    0 & D_3 \\
  \end{array}
\right).
$$

We thus have shown that $A=N+D$, and so $B=N+D+c\Id$, where $N^2=0$ and $D+c\Id$ is diagonalizable with eigenvalues $c,c+1,c+a,c+a+1$, as required.
\end{proof}

Our last consequence in this section is the following one.

\begin{corollary}\label{4k+3-corollary}
If $\FF=\FF_2$ and $B$ is a non-derogatory matrix of order $4k+3$ with $k\ge 1$ over $\FF$, then $B$ can be expressed as the sum of a square-zero matrix $N$ and a potent matrix $D$ with $D^4=D$.
\end{corollary}

\begin{proof}
The assertion follows from Proposition \ref{4k+3} by considering $B$ as a matrix over $\FF_4$ and taking $a\in \FF_4$ such that $a^2+a+1=0$; notice that $N$ and $D$ are matrices over $\FF_2$ mimicking the same arguments as in the proof of Case (b) of Corollary \ref{4k-corollary}. Moreover, one observes that $D^4=D$, because it is a diagonalizable matrix over $\FF_4$ with eigenvalues $0,1,a,a+1$.
\end{proof}

\section{Major Consequences}

As a culmination of our work in the preceding section, the main achievement here is the following statement.

\begin{theorem}\label{main} Let $\FF$ be a field of characteristic 2 with $\#\FF\ge4$. Then, every square matrix over $\FF$ can be expressed as the sum of a square-zero matrix and a diagonalizable matrix.
\end{theorem}

\begin{proof} Choose $A\in\mathbb{M}_n(\mathbb{F})$; in view of the rational canonical decomposition, $A$ is similar to the direct sum of the companion matrices of its invariant factors. We now consider the following four possibilities:

\medskip

\noindent -- Invariant factors of degree one are diagonalizable, so they can be expressed as the sum of the zero-matrix and themselves.

\medskip

\noindent -- The companion matrices of polynomials of degree 2, 3 and 4 are decomposed into the sum of a square-zero matrix and a diagonalizable matrix thanks to Propositions \ref{case2}, \ref{case3} and \ref{case4}.

\medskip

\noindent -- The companion matrices of polynomials of degree $4k$, where $k\ge 2$, are decomposed into the sum of a square-zero matrix and a diagonalizable matrix thanks to Proposition \ref{4k}.

\medskip

\noindent -- The companion matrices of polynomials of degree $4k+1$, $4k+2$ and $4k+3$, where $k\ge 1$, are decomposed into the sum of a square-zero matrix and a diagonalizable matrix thanks to Propositions \ref{4k+1}, \ref{4k+2} and \ref{4k+3}.

\medskip

This guarantees that all arguments necessary to establish the asserted decomposition are generally satisfied, as expected.
\end{proof}

We are aware that, if $\FF$ is a finite field of characteristic 2 and $\#\FF\ge 4$, then its cardinality is precisely $2^m$ for some $m\ge 2$. Since every diagonalizable matrix $D\in M_n(\FF)$ satisfies $D^{2^m}=D$, as a logical consequence of the above result, we may formulate the following.

\begin{corollary}\label{potent2^m}
Let $\FF$ be a finite field of cardinality $2^m$ with $m\ge 2$, and let $A$ be a matrix over $\FF$. Then, there exists a square-zero matrix $N$ and a $2^m$-potent matrix $D$ such that $A=N+D$.
\end{corollary}

The previous result is definitely {\it not} true for the field $\FF_2$ as it was shown by \ ~\v{S}ter in \cite{St1} and Shitov in \cite{Sh}. Nevertheless, if we admit that the potent part has index less than or equal to four, then the decomposition of every matrix over $\FF_2$ into a square-zero matrix and a $4$-potent matrix surprisingly holds always.

\begin{corollary}\label{F2-potent}
Let $\FF=\FF_2$, and let $A$ be a matrix over $\FF$. Then, there exists a square-zero matrix $N$ and a $4$-potent matrix $D$ such that $A=N+D$.
\end{corollary}

\begin{proof}
Take $A\in\mathbb{M}_n(\mathbb{F})$; in virtue of the rational canonical decomposition, $A$ is similar to the direct sum of the companion matrices of its invariant factors. We now consider the following three possibilities:

\medskip

\noindent -- Companion matrices of polynomials of degree one are 2-potent, so we are done.

\medskip

\noindent -- The companion matrices of polynomials of degree 2 can be expressed as the sum of a square-zero matrix and a 2-potent matrix with the help of Proposition \ref{case2}.

\medskip

\noindent -- Companion matrices of polynomials of degree greater than or equal to 3 can be written as the sum of a square-zero matrix and a 4-potent matrix with Corollaries \ref{case3-corollary}, \ref{case4-corollary}, \ref{4k-corollary}, \ref{4k+1-corollary}, \ref{4k+2-corollary} and \ref{4k+3-corollary} in mind.
\end{proof}

Note that the order of potency of the matrix $D$ in Corollary \ref{potent2^m} is exactly the cardinality of the finite field. In general, this index of potency {\it cannot} be reduced, because it is achieved for certain non-derogative matrices of order one. Nevertheless, for non-derogative matrices of order $n\ge 2$, we can curiously adjust the potency depending on the characteristic polynomial of the matrix as the next assertion states.

\begin{proposition}\label{adjusting}
Let $n\ge 2$. Let $\FF$ be a finite field of characteristic $2$ and cardinality $\ge 4$, and let $A\in M_n(\FF)$ be a non-derogative matrix over $\FF$ with characteristic polynomial $$p_A(x)=x^n+u_{n-1}x^{n-1}+u_{n-2}x^{n-2}+\dots.$$
\begin{itemize}
 \item[(i)] If $n=2$ and ${\rm Trace}(A)=0$, then there exist $N,D\in M_n(\FF)$ such that $A=N+D$ with $N^2=0$ and $D^3=D$.
 \item[(ii)] If $n=4k+2$ and ${\rm Trace}(A)=0$, then, for any subfield $\KK$ of $\FF$ such that $\KK\ne \FF_2$ and such that $u_{n-2}\in \KK$, it follows that $A=N+D$ with $N^2=0$ and $D^s=D$, where $s=\# \KK$.
 \item[(iii)] In the rest of the cases, for any subfield $\KK$ of $\FF$ such that $\KK\ne \FF_2$ and such that $u_{n-1}\in \KK$, it follows that $A=N+D$ with $N^2=0$ and $D^s=D$, where $s=\# \KK$.
\end{itemize}
\end{proposition}

\begin{proof} We study the different orders $n$ separately.

\medskip

$\bullet$ $n=2$. If ${\rm Trace}(A)=u_1\ne 0$, Proposition \ref{case2} enables us to write that $A=N+D$, where $D$ is diagonalizable with eigenvalues $0, u_1\in \KK$ and so $D^s=D$ for $s=\# \KK$. Otherwise, even if the calculations in Proposition \ref{case2} allow to express $A$ as the sum of a square-zero matrix and a diagonalizable matrix of the form $$D=\left(
                                                                                                                                      \begin{array}{cc}
                                                                                                                                        v & 0 \\
                                                                                                                                        0 & v \\
                                                                                                                                      \end{array}
                                                                                                                                    \right)
,$$ where $v^2=\det(A)\in \KK$, thus leading to $D^s=D$ with $s =\#\KK$, in this case the order of potency of the potent part can be improved due to the following decomposition:
$$A=
\left(
  \begin{array}{cc}
    0 & u \\
    1 & 0 \\
  \end{array}
\right)=\underbrace{\left(
          \begin{array}{cc}
            0 & 1+u \\
            0 & 0 \\
          \end{array}
        \right)}_{N'}+\underbrace{\left(
          \begin{array}{cc}
            0 & 1 \\
            1 & 0 \\
          \end{array}
        \right)}_{D'},
$$
where $(N')^2=0$ and $(D')^3=D'$.

\medskip

$\bullet$ $n=3$. Suppose that ${\rm Trace}(A)=u_2\in \KK$, where $\KK$ is a subfield of $\FF$ with $\KK\ne \FF_2$. Applying Proposition \ref{case3}, we write $A=N+D$, where $N^2=0$ and $D$ is diagonalizable with eigenvalues $$a, a+1, u_2+a$$ for any $a\in\KK$ such that $a\ne u_2, u_2+1$. Hence, $D^s=D$ for $s=\#\KK$.

\medskip

$\bullet$ $n=4$. Suppose that $0\ne {\rm Trace}(A)=u_3\in \KK$, where $\KK$ is a subfield of $\FF$ with $\KK\ne \FF_2$. Employing Proposition \ref{case4}, we write $A=N+D$, where $N^2=0$ and $D$ is diagonalizable with eigenvalues $$0, a, u_3+a$$ for any $a\in\KK$ such that $a\ne 0, u_3$, whence $D^s=D$ for $s=\#\KK$. On the other hand, if ${\rm Trace}(A)=0$, Proposition \ref{case4} works to get that $A=N+D$, where $N^2=0$ and $D$ is diagonalizable with eigenvalues $$0,1,a,a+1,$$ and $a\in \KK$ with $a\ne, 0,1$ for any subfield $\KK$ of $\FF$ with $\KK\ne \FF_2$. Hence, $D^s=D$ for $s=\#\KK$.

\medskip

$\bullet$ $n=4k$, $k\ge 2$. Suppose that $0\ne {\rm Trace}(A)=u_{n-1}\in \KK$, where $\KK$ is a subfield of $\FF$ with $\KK\ne \FF_2$. Utilizing Proposition \ref{4k}, we write $A=N+D$, where $N^2=0$ and $D$ is diagonalizable with eigenvalues $$0,u_{n-1}, u_{n-1}\,a, u_{n-1}\, (a+1)$$ for any $a\in\KK$ such that $a\ne 0,1$, whence $D^s=D$ for $s=\#\KK$. On the other hand, if ${\rm Trace}(A)=0$, Proposition \ref{4k} applies to get that $A=N+D$, where $N^2=0$ and $D$ is diagonalizable with eigenvalues $$0,1,a,a+1$$ such that $\KK$ is a subfield of $\FF$ with $\KK\ne \FF_2$, and $a\in \KK$ with $a\ne 0,1$. Hence, $D^s=D$ for $s=\#\KK$.

\medskip

$\bullet$ $n=4k+1$, $k\ge 1$. Suppose that ${\rm Trace}(A)=u_{n-1}\in \KK$, where $\KK$ is a subfield of $\FF$ with $\KK\ne \FF_2$. By Proposition \ref{4k+1}, we write $A=N+D$, where $N^2=0$ and $D$ is diagonalizable with eigenvalues $$u_{n-1}, u_{n-1}+1, u_{n-1}+a, u_{n-1}+a+1,$$ and $a\in \KK$ with $a\ne 0,1$. Hence, $D^s=D$ for $s=\#\KK$.

\medskip

$\bullet$ $n=4k+2$, $k\ge 1$. Suppose that $0\ne {\rm Trace}(A)=u_{n-1}\in \KK$, where $\KK$ is a subfield of $\FF$ with $\KK\ne \FF_2$. According to Proposition \ref{4k+2}, we write $A=N+D$, where $N^2=0$ and $D$ is diagonalizable with eigenvalues $$0,u_{n-1}, u_{n-1}\,a, u_{n-1}\, (a+1)$$ for any $a\in\KK$, $a\ne 0,1$, whence $D^s=D$ for $s=\#\KK$. On the other hand, if ${\rm Trace}(A)=0$, we consider the coefficient $u_{n-2}$ of degree $n-2$ in $p_{A}(x)$, and let $\KK$ be any subfield of $\FF$ with $\KK\ne \FF_2$ such that $u_{n-2}\in \KK$. Taking into account  Proposition \ref{4k+2}, we write $A=N+D$, where $N^2=0$ and $D$ is diagonalizable with eigenvalues $$b, b+1, b+a, b+a+1$$ such that $a\in \KK$, $a\ne 0,1$, and $$b^2=u_{n-2}+a^2+a+1;$$  it is worthy of noticing that $b\in \KK$, because the map $\phi:\KK\to \KK$ defined by $\phi(x)=x^2$ is surjective. Hence, $D^s=D$ for $s=\#\KK$.

\medskip

$\bullet$ $n=4k+3$, $k\ge 1$. Suppose that ${\rm Trace}(A)=u_{n-1}\in \KK$, where $\KK$ is a subfield of $\FF$ with $\KK\ne \FF_2$. Referring to Proposition \ref{4k+3}, we write $A=N+D$, where $N^2=0$ and $D$ is diagonalizable with eigenvalues $$u_{n-1}, u_{n-1}+1, u_{n-1}+a, u_{n-1}+a+1$$ such that $a\in \KK$ with $a\ne 0,1$. Hence, $D^s=D$ for $s=\#\KK$.

\medskip

This exhausts all possible variants, and thus terminates the whole argumentation, as needed.
\end{proof}

\bigskip
\bigskip
\bigskip

\noindent{\bf Funding:} The first-named author (Peter Danchev) was supported in part by the BIDEB 2221 of T\"UB\'ITAK; the second and third-named authors (Esther Garc\'{\i}a and Miguel G\'omez Lozano) were partially supported by B42025-003: Ayudas para Proyectos Puente de la UMA. All three authors were partially supported by the Junta de Andaluc\'{\i}a PPRO-FQM264-G-2023 (FQM264-G-FEDER).

\vskip3.0pc

\bibliographystyle{plain}

\begin{thebibliography}{100}

\bibitem{Br}
S.~Breaz.
\newblock Matrices over finite fields as sums of periodic and nilpotent elements.
\newblock {\em Linear Algebra \& Appl.}, {\bf 555}:92--97, 2018.

\bibitem{BCDM}
S.~Breaz, G.~C\v{a}lug\v{a}reanu, P.~Danchev and T.~Micu.
\newblock Nil-clean matrix rings.
\newblock {\em Linear Algebra \& Appl.}, {\bf 439}:3115--3119, 2013.

\bibitem{BM}
S.~Breaz and S.~Megiesan.
\newblock Nonderogatory matrices as sums of idempotent and nilpotent matrices.
\newblock {\em Linear Algebra \& Appl.}, {\bf 605}:239--248, 2020.

\bibitem{DGL1}
P.~Danchev, E.~Garc\'ia and M.~G\'omez Lozano.
\newblock Decompositions of matrices into diagonalizable and square-zero matrices.
\newblock {\em Linear and Multilinear Algebra}, {\bf 70}(19):4056--4070, 2022.

\bibitem{DGL2}
P.~Danchev, E.~Garc\'ia and M.~G. Lozano.
\newblock Decompositions of matrices into potent and square-zero matrices.
\newblock {\em Internat. J. Algebra Comput.}, {\bf 32}(2):251--263, 2022.

\bibitem{DGL3}
P.~Danchev, E.~Garc\'ia and M.~G. Lozano.
\newblock Decompositions of matrices into a sum of torsion matrices and matrices of fixed nilpotence.
\newblock {\em Linear Algebra Appl.}, {\bf 676}:44--55, 2023.

\bibitem{DGL4}
P.~Danchev, E.~Garc\'ia and M.~G. Lozano.
\newblock Matrices over finite fields of odd characteristic as sums of diagonalizable and square-zero matrices
\newblock {\em Linear Algebra \& Appl.}, {\bf 730}:35--50, 2026.

\bibitem{Sh}
Y.~Shitov.
\newblock The ring $\mathbb{M}_{8k+4}(\mathbb{Z}_2)$ is nil-clean of index four.
\newblock {\em Indag. Math. (N.S.)}, {\bf 30}:1077--1078, 2019.

\bibitem{St1}
J.~\v{S}ter.
\newblock On expressing matrices over $\mathbb{Z}_2$ as the sum of an idempotent and a nilpotent.
\newblock {\em Linear Algebra \& Appl.}, {\bf 544}:339--349, 2018.

\bibitem{St2}
J.~\v{S}ter.
\newblock Nil-clean index of $\mathbb{M}_n(\mathbb{F}_2)$.
\newblock {\em Linear Algebra \& Appl.}, {\bf 632}:294--307, 2022.

\end{thebibliography}

\end{document}